\documentclass[10pt,reqno]{amsart}
\usepackage{amsmath,amssymb,amsthm}
\usepackage{graphicx}
\usepackage[all]{xypic}
\usepackage[utf8,nocaptions]{vietnam}
\usepackage{hyperref}
\input xypic
\def\thmsection{section}
\def\thmchangesection{changesection}

\def\thmchangechapter{changechapter}
\def\thmchange{change}
\def\thmplain{plain}
\ifx\theoremnumberstyle\thmplain
  \theoremstyle{break-italic}
  \newtheorem{satz}{Satz}
\else
  \ifx\theoremnumberstyle\thmsection
    \theoremstyle{break-italic}
    \newtheorem{satz}{Satz}[section]
  \else
    \ifx\theoremnumberstyle\thmchange
      \swapnumbers
      \theoremstyle{break-italic}
      \newtheorem{satz}{Satz}
    \else
      \ifx\theoremnumberstyle\thmchangesection
         \swapnumbers
         \theoremstyle{break-italic}
         \newtheorem{satz}{Satz}[section]
      \else
        \ifx\theoremnumberstyle\thmchangechapter
           \swapnumbers
           \theoremstyle{break-italic}
           \newtheorem{satz}{Satz}[chapter]
        \else
          \ifx\theoremnumberstyle\thmsection
             \theoremstyle{break-italic}
             \newtheorem{satz}{Satz}[section]
          \else
            \theoremstyle{break-italic}
            \newtheorem{satz}{Satz}[section]
          \fi
        \fi
      \fi
    \fi
  \fi
\fi

\theoremstyle{break-italic}
\newtheorem{theorem}[satz]{Theorem}
\newtheorem{lemma}[satz]{Lemma}

\newtheorem{corollary}[satz]{Corollary}

\newtheorem*{conjecture*}{Conjecture}

\theoremstyle{break-roman}
\newtheorem{definition}[satz]{Definition}

\newtheorem{remark}[satz]{Remark}

\theoremstyle{standard}

\newtheorem{claim}[satz]{Claim}
\theoremstyle{varthm-roman}
\newtheorem*{varthm-roman}{}
\theoremstyle{varthm-italic}
\newtheorem*{varthm-italic}{}
\theoremstyle{varthm-roman-break}
\newtheorem*{varthm-roman-break}{}
\theoremstyle{varthm-italic-break}
\newtheorem*{varthm-italic-break}{}
\theoremstyle{varthm-roman-no-punctuation}
\newtheorem{varthm-roman-no-punctuation-numbered}[satz]{}
\theoremstyle{varthm-italic-no-punctuation}
\newtheorem{varthm-italic-no-punctuation-numbered}[satz]{}

\newenvironment{varthm-roman-numbered}[1]{
  \begin{varthm-roman-no-punctuation-numbered}
    \mbox{\rm\textbf{#1}}
  }{\end{varthm-roman-no-punctuation-numbered}}
\newenvironment{varthm-italic-numbered}[1]{
  \begin{varthm-italic-no-punctuation-numbered}
    \mbox{\rm\textbf{#1}}
  }{\end{varthm-italic-no-punctuation-numbered}}
\newenvironment{varthm-roman-break-numbered}[1]{
  \begin{varthm-roman-no-punctuation-numbered}
    \mbox{\rm\textbf{#1}\newline}
  }{\end{varthm-roman-no-punctuation-numbered}}
\newenvironment{varthm-italic-break-numbered}[1]{
  \begin{varthm-italic-no-punctuation-numbered}
    \mbox{\rm\textbf{#1}}\newline
  }{\end{varthm-italic-no-punctuation-numbered}}
\numberwithin{equation}{section}

\newtheorem{proposition}{Proposition}[section]

\usepackage{chngcntr}
\everymath{\displaystyle}
\begin{document}
\title[Perturbation of monic matrix polynomials]{Perturbation of monic matrix polynomials}
\author{Cong Trinh Le} 
\email{lecongtrinh@qnu.edu.vn}
\address{Department of Mathematics and Statistics, Quy Nhon University, Quy Nhon Nam Ward, Vietnam}
\author{Gue Myung Lee}
\email{gmlee@pknu.ac.kr}
\address{Department of Applied Mathematics, Pukyong National University, Busan 48513, Korea}
\author{Yongdo Lim}
\email{ylim@skku.edu}
\address{Department of Mathematics, Sungkyunkwan University, Suwon 440--746, Korea}
\author{Tien Son Pham}
\email{sonpt@dlu.edu.vn}
\address{Department of Mathematics, University of Dalat, 1 Phu Dong Thien Vuong, Dalat, Vietnam}
\date{\today}
\begin{abstract}
In this paper, we study the stability of matrix polynomials under structured
perturbations of their coefficients. More precisely, we consider a family of
matrix polynomials
\[
P_u(\lambda)=A_d(u)\lambda^d+A_{d-1}(u)\lambda^{d-1}+\cdots+A_0(u),
\]
whose matrix coefficients depend continuously and semialgebraically on a
parameter vector $u\in\mathbb{C}^p$.

Assuming that the matrix polynomial is monic, we show that the spectrum, the
$\varepsilon$-pseudospectrum, the numerical range, and the joint numerical
range associated with $P_u(\lambda)$ define set-valued maps that are
H\"older continuous with respect to the parameter $u$. Moreover, the
parameter space $\mathbb{C}^p$ can be decomposed into a finite union of
analytic semialgebraic submanifolds such that, on each submanifold, the
eigenvalues and the Jordan pairs of $P_u(\lambda)$ depend analytically on $u$.

We also note that most of the results remain valid if the monicity assumption
is replaced by the local nonsingularity of the leading coefficient matrix
$A_d(u)$. However, the monic setting is adopted throughout the paper in order
to simplify the exposition and to avoid additional technical assumptions,
which are required in particular for results concerning numerical ranges.
\end{abstract}
\keywords{Matrix polynomial; structured perturbation; spectrum;  $\epsilon$-pseudospectrum; numerical range; Jordan pair; stability; semialgebraic set}
\subjclass[2000]{15A18; 15A22; 65F15; 65F35; 14P10; 32B20}

\maketitle
\section{Introduction}

Let $\mathbb{C}^{n\times n}$ be the algebra of all $n \times n$ complex matrices and consider the \emph{perturbed matrix polynomial} of the form
\begin{eqnarray*}
P_{u}(\lambda) &:=& A_d(u)  \lambda^d  + A_{d - 1}(u) \lambda^{d - 1}  + \cdots + A_{0}(u),
\end{eqnarray*}
where $\lambda \in \mathbb{C}$  and $A_{k}(u) \in \mathbb{C}^{n\times n}$ for $u \in \mathbb{C}^p.$  The study of matrix (and operator) polynomials has a long history, especially with regard to their spectral analysis and numerical ranges; for more details, we refer the reader to \cite{Gantmacher1959, Gohberg1982, Lancaster1966, Markus1988}.

Let us begin with some formal definitions. First, for each parameter $u \in \mathbb{C}^p,$ the set
\begin{eqnarray*}
\sigma(u) &:=& \{ \lambda \in \mathbb{C} \ | \ \det P_{u}(\lambda) = 0\}
\end{eqnarray*}
is called the \emph{spectrum} of the matrix polynomial $P_u(\lambda)$. Given a nonnegative real number $\epsilon,$ the \emph{$\epsilon$-pseudospectrum} of $P_u(\lambda)$ is defined by
\begin{eqnarray*}
\Lambda_{\epsilon}(u) &:=& \{\lambda \in \mathbb C \ | \ \det P_{v}(\lambda)=0 \mbox{ for some } v \in \mathbb C^p \mbox{ with } \|v - u\|\leq \epsilon\}.
\end{eqnarray*}
Our definition of the $\varepsilon$-pseudospectrum corresponds to a structured  pseudospectrum induced by perturbations of the parameter vector $u$. While the classical pseudospectrum considers arbitrary perturbations of bounded  norm, our definition restricts perturbations to those arising from admissible parameter variations. This framework is consistent with structured pseudospectra studied in, e.g., control and polynomial eigenvalue problems \cite{Tisseur2001-1, Higham2003}.

Recall that for a matrix $A \in \mathbb{C}^{n \times n}$, its numerical range is defined as
$
W(A) = \{x^*Ax : x \in \mathbb{C}^n, \ x^*x = 1\}.
$
The \emph{numerical range} of $P_u(\lambda)$ is defined by
\begin{eqnarray*}
W(u) &:=& \{\lambda \in \mathbb{C} \ | \ x^*P_u(\lambda)x =0 \mbox{ for some }  x\in \mathbb C^n, x^*x=1\},
\end{eqnarray*}
which contains $\sigma(u).$ Here and in the following the symbol $x^*$ stands for conjugate transpose of $x.$ An important extension of the classical numerical range is the  \emph{joint numerical range} of the $(d + 1)$-tuple of the matrices $A_0(u), \ldots, A_{d}(u)$ which is defined by
\begin{eqnarray*}
JW(u) &:=& \{(x^* A_0(u) x, \ldots, x^*A_{d}(u)x) \ | \   x \in \mathbb{C}^n, x^*x=1\}.
\end{eqnarray*}

 Spectra, $\epsilon$-pseudospectra, numerical ranges and joint numerical ranges of matrix polynomials are useful in the study of various pure and applied subjects; see, for example, \cite{Armentia2017, Binding1991, Boulton2008, Doyle1982, Duffin1955, Fan1988, Gau2011, Gohberg1982, Gustafson1997, Gutkin2004, Higham2002, Higham2003, Horn2013, Jonckheere1998, Lancaster1966, Lancaster2001, Lancaster2003, Lancaster2005, Li1994, Li1999, Li2009, Maroulas1997, Mosier1986, Poon1997, Psarrakos2000, Tisseur2001-1, Tisseur2001-2, Trefethen1992, Trefethen1999-1, Trefethen1999-2, Zeng2014} and the references therein.

Inspired by the works in \cite{Boulton2008, Lancaster2005, Li1994, Tisseur2001-1}, in this paper we are principally interested in the following question, which is important in application: \textit{How stable are the spectrum $\sigma(u),$ the $\epsilon$-pseudospectrum $\Lambda_\epsilon(u),$ the numerical range $W(u),$ and the joint numerical range $JW(u)$ of the matrix polynomial $P_u(\lambda)$ when we perturb its coefficients $A_k(u),$  $ k = 0, 1,\ldots, d, $ in $u\in \mathbb C^p$?}  In other words, \textit{how stable are the following  set-valued maps?}
\begin{itemize}
\item The spectrum set-valued map
$$\sigma \colon \mathbb{C}^p \rightrightarrows \mathbb{C}, \quad u \mapsto \sigma(u);$$

\item  The  $\epsilon$-pseudospectrum set-valued map
$$\Lambda_\epsilon  \colon \mathbb{C}^p \rightrightarrows \mathbb{C}, \quad u \mapsto \Lambda_\epsilon(u);$$

\item  The numerical range set-valued map
 $$W \colon \mathbb{C}^p \rightrightarrows \mathbb{C}, \quad u \mapsto W(u);$$

\item  The joint numerical range set-valued map
$$JW \colon \mathbb{C}^p \rightrightarrows \mathbb{C}^{d + 1}, \quad u \mapsto JW(u).$$
\end{itemize}

Assume that the matrix polynomial $P_u(\lambda)$ is {\em monic} (i.e., $A_{d}(u) \equiv I_d$-the identity matrix) with  its matrix coefficients $A_{k}(u), k = 0, 1, \ldots, d - 1, $ being \emph{continuous semialgebraic} on the parameter space $\mathbb C^p.$ By exploiting results from semi-algebraic geometry, we will show  that (i) the above set-valued  maps are H\"older continuous and (ii) there exists an open and dense semialgebraic subset of $\mathbb C^p$ over which the maps $\sigma$, $\Lambda_\epsilon$ and $W$ are H\"older continuous  with an explicit degree, while the map $JW$ is Lipschitz continuous. Perturbations of matrix polynomials are not arbitrary but are induced by  variations of a parameter vector $u \in \mathbb{C}^p$. This corresponds to structured perturbations of the matrix coefficients, which naturally arise in applications where coefficients depend on physical, geometric, or design parameters. Our analysis focuses on the stability of spectral objects under such structured perturbations. 

The paper is organized as follows. In Section~\ref{SecPreliminaries} we recall some basic notions and results in semialgebraic geometry. We present  in Section~\ref{Section-Stability} the H\"older continuity of  the set-valued maps $\sigma, \Lambda_\epsilon, W, ~\partial W$ and $JW$, and of the spectral radius
$$\mathfrak{r}(u):= \max \{|\lambda|, \lambda \in \sigma(u)\}.$$
Moreover, we also  give in this section some geometric properties of the $\epsilon$-pseudospectrum and its boundary. Finally, it is shown in Section~\ref{jordanpair} that the parameter space $\mathbb{C}^p$ can be decomposed into a finite union of analytic semialgebraic submanifolds such that the restrictions of the eigenvalues and the Jordan pairs of the monic matrix polynomial $P_u(\lambda)$ on each submanifold  are analytic.

\begin{remark}\rm
 In general, the assumption that the matrix polynomial $P_u(\lambda)$ is monic can be relaxed. 
Most of the results of this paper remain valid if the leading coefficient $A_d(u)$ is merely 
\emph{locally nonsingular}, that is, if for each $\tilde u \in \mathbb{C}^p$ there exists a 
neighborhood $U$ of $\tilde u$ such that $A_d(u)$ is nonsingular for all $u \in U$.

Nevertheless, we restrict ourselves to the monic case $A_d(u)\equiv I_d$ in order to simplify the 
exposition and to avoid additional technical assumptions. In some sections--most notably those 
concerning numerical ranges--stronger conditions on the leading coefficient are implicitly 
required to guarantee boundedness and stability of the corresponding sets. These additional 
assumptions are stated explicitly where needed.
\end{remark}

\subsection*{Notation}
For each $p\in \mathbb N$, $p>0$, the space $\mathbb{C}^p$ is equipped with a norm, denoted by $\| \cdot\|.$  For a given $u \in \mathbb{C}^p$ and a given $\delta > 0,$ the closed ball in $\mathbb{C}^p$ with center $x$ and radius $\delta$ is denoted by $\mathbb{B}(u, \delta).$    For a matrix $A \in \mathbb C^{n\times n},$ $\|A\|$ denotes a subordinate matrix norm of $A.$

\section{Preliminaries on  semialgebraic geometry} \label{SecPreliminaries}

Let us recall some basic notions and results on semialgebraic geometry (see, for example, \cite{Bochnak1998}) which we need.

\begin{definition}{\rm
\begin{enumerate}
  \item[(i)] A subset of $\mathbb{R}^n$ is called {\em semialgebraic} if it is a finite union of sets of the form
$$\{x \in \mathbb{R}^n \ | \ f_i(x) = 0, i = 1, \ldots, k; f_i(x) > 0, i = k + 1, \ldots, p\},$$
where all $f_{i}$ are polynomials.
 \item[(ii)]
Let $A \subset \Bbb{R}^n$ and $B \subset \Bbb{R}^m$ be semialgebraic sets. A map $F \colon A \to B$ is said to be {\em semialgebraic} if its graph
$$\{(x, y) \in A \times B \ | \ y = F(x)\}$$
is a semialgebraic subset in $\Bbb{R}^n\times\Bbb{R}^m.$
\end{enumerate}
}\end{definition}

The class of semialgebraic sets is closed under taking finite intersections, finite unions, and complements; a Cartesian product of semialgebraic sets is a semialgebraic set. Moreover, a major fact concerning the class of semialgebraic sets is its stability under linear projections (see, for example, \cite{Bochnak1998}).

\begin{theorem}[Tarski--Seidenberg theorem] \label{TarskiSeidenbergTheorem}
The image of a semialgebraic set by a semialgebraic map is semialgebraic.
\end{theorem}

By the Tarski--Seidenberg theorem, it is not hard to see that the closure, the interior, and the boundary of a semialgebraic set are semialgebraic sets.

\begin{definition}{\rm
Let $A \subset {\mathbb R}^n$ and $f \colon A \rightarrow {\mathbb R}^m.$ The map $f$ is said to be {\em analytic at a point $p \in A$} if there exists an open neighborhood $U$ of $p$ in ${\mathbb R}^n$ and an analytic map $\hat{f} \colon U \rightarrow {\mathbb R}^m$ such that the restriction $\hat{f}|_{A \cap U}$ is equal to $f|_{A \cap U}.$ The map $f$ is said to be {\em analytic} on $A$ iff it is analytic at every point in $A.$
}\end{definition}

\begin{definition} {\rm
A nonempty subset $A$ of ${\mathbb R}^n$ is an {\em analytic  $k$-dimensional submanifold} of ${\mathbb R}^n$, where $1 \le k \le n,$ if for each point $p$ in $A$ there exists an open neighborhood $W$ of $p$ in ${\mathbb R}^n$ and an open set $U$ in ${\mathbb R}^k$ such that after some permutation of coordinates we have
$$A \cap W = \text{graph}\, \phi := \{(u, \phi(u)) \, | \, u \in U\},$$
where $\phi \colon U \to {\mathbb R}^{n - k}$ is an analytic map.

By convention, a point in ${\mathbb R}^n$ is an analytic 0-dimensional submanifold of ${\mathbb R}^n.$
}\end{definition}

By the Cell Decomposition Theorem (see \cite{Bochnak1998}), for any semialgebraic subset $A \subset \mathbb{R}^n,$ we can write $A$ as a disjoint union of finitely many analytic semialgebraic manifolds of different dimensions.  The {\em dimension} $\dim A$ of a nonempty semi-algebraic set $A$ can thus be defined as the dimension of the manifold of highest dimension of its decomposition. This dimension is well defined and independent of the decomposition of $A.$ By convention, the dimension of the empty set is taken to be negative infinity. We will need the following result (see \cite{Bochnak1998}).

\begin{proposition} \label{DimensionProposition}
If $A \subset \mathbb{R}^n$ is a nonempty semialgebraic set, then the boundary $\partial {A}$ of $A$ is also a semialgebraic set and it holds that
$$\dim \partial {A} < \dim A.$$
\end{proposition}

\begin{proposition}\label{LojasiewiczInequality}
Let $f \colon \mathbb{R}^n \rightarrow \mathbb{R}^m$ be a continuous semialgebraic map. Then for any compact $K \subset \mathbb{R}^n,$ there exist constants $c > 0$ and $\alpha > 0$ such that
\begin{eqnarray*}
\| f(x) - f(x') \| & \le & c \|x - x'\|^\alpha \quad \textrm{ for all } \quad x, x' \in K.
\end{eqnarray*}
\end{proposition}

\begin{proof}
This is an immediate consequence of the {\L}ojasiewicz inequality (see, for example, \cite[Theorem~1.14]{HP2017}).
\end{proof}

\begin{remark}{\rm
In general, it is difficult to obtain explicit formulas for the \L ojasiewicz exponent $\alpha$ in the above proposition; for more details on this topic, see \cite{Dinh2016, Dinh2017, Kurdyka2014, Pham2012}.
}\end{remark}

  By identifying $\mathbb{C}$ with $\mathbb{R}^2$ we can define semialgebraic sets in $\mathbb{C}^n$ and semialgebraic maps from
semialgebraic sets in $\mathbb{C}^n$ to $\mathbb{C}^m.$ As a consequence, the results in this subsection still hold if we replace the real field $\mathbb{R}$ by the complex filed $\mathbb{C}.$

\section{Stability of set-valued maps} \label{Section-Stability}

Throughout the rest of this paper we let
$$P_{u}(\lambda) := A_d(u) \lambda^d  + A_{d - 1}(u) \lambda^{d - 1}  + \cdots + A_{0}(u)$$
be a matrix polynomial, where $\lambda$ is a complex variable and $A_{k}(u), k = 0, \ldots, d,$ are complex $n \times n$ matrices whose entries are continuous semialgebraic functions in $u \in \mathbb{C}^p.$

\begin{lemma}\label{Lemma31}
Let $\tilde{u} \in \mathbb{C}^p.$ For  sufficiently small $\delta > 0$ there exist constants $c > 0$ and $\alpha > 0$ such that
$$\max_{k = 0, \ldots, d} \|A_k(u) - A_k(u')\| \le c\|u - u'\|^\alpha \quad \mbox{ for all } \quad u, u' \in \mathbb{B}(\tilde{u}, \delta).$$
\end{lemma}

\begin{proof}
This follows from the continuity and semialgebraicity of the matrices $A_k(u)$, $k=0, \ldots, d,$ and Proposition~\ref{LojasiewiczInequality}.
\end{proof}

 \begin{lemma}\label{Lemma32}
Let $\tilde{u} \in \mathbb{C}^p.$ For  sufficiently small $\delta > 0$ and any compact $K \subset \mathbb{C}$ there exist constants $c > 0$ and $\alpha > 0$ such that
$$\| P_{u}(\lambda) - P_{u'}(\lambda) \| \le c \|u - u'\|^\alpha \quad \textrm{ for all } \quad
\lambda \in K \textrm{ and all } u, u' \in \mathbb{B}(\tilde{u}, \delta).$$
\end{lemma}
\begin{proof}
Let $\delta >0$ and $K \subseteq \mathbb C$ compact be given.  We have for all $\lambda \in K$ and all  $u, u' \in \mathbb{B}(\tilde{u}, \delta),$
\begin{eqnarray*}
\| P_{u}(\lambda) - P_{u'}(\lambda) \| & \le & \sum_{k = 0}^{d} \|A_k(u) - A_k(u')\| |\lambda|^k \\
& \le & \left( \max_{k = 0, \ldots, d} \|A_k(u) - A_k(u')\| \right) \sum_{k = 0}^{d} |\lambda|^k \\
& \le & \left( \max_{k = 0, \ldots, d} \|A_k(u) - A_k(u')\| \right) \sum_{k = 0}^{d} c_0^k,
\end{eqnarray*}
where $c_0 := \max  \{|\lambda| \ | \ {\lambda \in K}\} < + \infty.$ Then the desired conclusion follows from Lemma~\ref{Lemma31}.
\end{proof}

\begin{definition}{\rm
A set-valued map\footnote{Recall that ${\mathcal S}$ is a {\em set-valued map} or a {\em multifunction} from $X$ to $Y,$ denoted by ${\mathcal S} \colon X  \rightrightarrows Y,$ if, for every $u \in X,$ ${\mathcal S}(u)$ is a subset of $Y.$} ${\mathcal S} \colon \mathbb{C}^{p} \rightrightarrows \mathbb{C}^{q}$ is said to be {\em H\"older continuous of degree $\alpha$} at $\tilde{u} \in \mathbb{C}^p$ when there are a constant $c > 0$ and an open  neighborhood $U \subset \mathbb{C}^n$ of $\tilde{u}$ such that
$$\mathcal{S}(u) \subset \mathcal{S}(u') + c \|u - u'\|^{\alpha} \mathbb{B} \quad \textrm{ for all } \quad u, u' \in U,$$
where $\mathbb{B}$ stands for the unit closed ball in $\mathbb{C}^{q}.$
}\end{definition}

In what follows we assume that $P_u(\lambda)$ is a {\em monic} matrix polynomial (i.e., $A_d(u)$ is the identity matrix $I$ in $\mathbb C^{n\times n}$), unless stated otherwise, and study the H\"older continuity of the spectrum,  (resp., the $\epsilon$-pseudospectrum, the numerical range, and the joint numerical range) set-valued map.

\subsection{The continuity of the spectrum set-valued map}

For each $u \in \mathbb{C}^p,$  we recall that the {\em spectrum} of  the matrix polynomial $P_{u}(\lambda)$ is defined as
\begin{eqnarray*}
\sigma(u) := \{ \lambda \in \mathbb{C} \ | \ \det P_{u}(\lambda) = 0\}.
\end{eqnarray*}
 Note that $\det P_{u}(\lambda)$ is a monic polynomial of degree $dn$ in the variable $\lambda,$   so
the set $\sigma(u)$ has at most $dn$ points.

\begin{lemma}\label{Lemma33}
Let $\tilde{u} \in \mathbb{C}^p.$ For sufficiently small $\delta > 0$ there exists a compact $K \subset \mathbb{C}$ such that
$$\sigma(u) \subset K \quad \textrm{ for all } \quad u \in \mathbb{B}(\tilde{u}, \delta).$$
\end{lemma}
\begin{proof}
By definition, $\det P_u(\lambda) \in \mathbb{C}[\lambda]$ is a monic polynomial whose coefficients depend continuously on the parameter $u \in \mathbb{C}^p.$ Consequently, the roots of
the equation $\det P_u(\lambda) = 0$ are continuous functions in $u.$ Then the desired result follows since the closed ball $\mathbb{B}(\tilde{u}, \delta)$ is compact.
\end{proof}

\begin{theorem}\label{Theorem32}
Let $\tilde{u} \in \mathbb{C}^p.$ For sufficiently small $\delta > 0$ there exist constants $c > 0$ and $\alpha > 0$ such that
\begin{eqnarray*}
\sigma(u) &\subset& \sigma(u') + c \|u - u'\|^\alpha \mathbb{B} \quad \textrm{ for all } \quad u, u' \in \mathbb{B}(\tilde{u}, \delta).
\end{eqnarray*}
In particular, the spectrum set-valued map
$$\sigma \colon \mathbb{C}^p \rightrightarrows \mathbb{C}, \quad u \mapsto \sigma(u),$$
is H\"older continuous of degree $\alpha$ at $\tilde{u}.$
\end{theorem}

\begin{proof}
Fix a real number $\delta > 0.$ By Lemma~\ref{Lemma33}, there exists a compact set $K \subset \mathbb{C}$ such that
$$\sigma(u) \subset K \quad \textrm{ for all } \quad u \in \mathbb{B}(\tilde{u}, \delta).$$
Thanks to Lemma~\ref{Lemma32}, we can find constants $c_0 > 0$ and $\alpha_0 > 0$ such that
\begin{eqnarray} \label{inequality1}
\| P_{u}(\lambda) - P_{u'}(\lambda) \| & \le & c_0 \|u - u'\|^{\alpha_0} 
\end{eqnarray}
for all $ \lambda \in K $  and all  $ u, u' \in \mathbb{B}(\tilde{u},\delta)$.\\
Furthermore, since $P_u(\lambda)$ is a polynomial matrix in $\lambda$ whose coefficients are continuous functions in $u$ and the sets $K$ and $\mathbb{B}(\tilde{u}, \delta)$ are compact, there exists a constant $c_1 > 0$ such that
\begin{eqnarray} \label{inequality2}
\| P_{u}(\lambda) \| & \le & c_1 \quad \textrm{ for all } \quad \lambda \in K \textrm{ and all } u \in \mathbb{B}(\tilde{u}, \delta).
\end{eqnarray}

Take any $u, u' \in \mathbb{B}(\tilde{u}, \delta)$ and any $\lambda \in \sigma(u).$ Since $\sigma(u')$ is the set of roots of the monic polynomial $\det P_{u'}(z) \in \mathbb C[z]$ whose degree is $dn$, it is easy to see that\footnote{More generally, for a polynomial $p\in \mathbb C[z]$ of degree $d,$ the following inequality holds
$$  |a_d|\mathrm{dist}(z, p^{-1}(0))^{d}\leq  |p(z)| \quad \textrm{ for all } \quad z \in \mathbb C, $$
where $a_d$ denotes the coefficient of the term of highest degree of $p.$ To see this, we write
$p(z) = a_d \prod_{k = 1}^d (z - z_k),$ where $z_1, \ldots, z_d$ are (not necessarily distinct) roots of $p.$ Then for any $z\in \mathbb C,$ we have
$$ |p(z)| = |a_d| \prod_{k = 1}^d |z - z_k| \geq |a_d| \prod_{k = 1}^d \min_{j = 1, \cdots, d} |z - z_j| = |a_d| \mathrm{dist}(z, p^{-1}(0))^d. $$}
\begin{eqnarray*}
\mathrm{dist}(z, \sigma(u'))^{dn} & \leq & |\det P_{u'}(z)| \quad \textrm{ for all } \quad z \in \mathbb C,
\end{eqnarray*}
where $\mathrm{dist}(\cdot, \cdot)$ stands for the distance function to a set. In particular,
\begin{equation} \label{inequality3}
 \mathrm{dist}(\lambda, \sigma(u'))^{dn} \leq |\det P_{u'}(\lambda)|.
 \end{equation}
But $\det P_{u}(\lambda) = 0$ because of $\lambda \in \sigma(u).$ Hence
\begin{eqnarray*}
|\det P_{u'}(\lambda)|
& = & |\det P_{u}(\lambda) - \det P_{u'}(\lambda)| \\
& \leq & n \max\{  \|P_{u}(\lambda)\|, \|P_{u'}(\lambda)\|\}^{n - 1}\cdot \|P_{u}(\lambda) - P_{ u'}(\lambda)\|,
\end{eqnarray*}
where the inequality follows from \cite[Corollary~20.2]{Bhatia2007}). Consequently, we deduce from the inequalities (\ref{inequality1}), (\ref{inequality2}) and (\ref{inequality3}) that
\begin{eqnarray*}
\mathrm{dist}(\lambda, \sigma(u'))^{dn} & \leq & n c_1^{n - 1} c_0 \|u - u'\|^{\alpha_0},
\end{eqnarray*}
and then the theorem follows easily.
\end{proof}

\begin{remark} \label{rem:theorem32} \rm Theorem \ref{Theorem32} is still true if we replace the hypothesis on monicity of the matrix polynomial $P_u(\lambda)$ by the locally non-singularity of the leading coefficient matrix $A_d(u)$ as follows.

\textit{For each $\tilde{u}\in \mathbb C^p$, there exists $\delta > 0$ such that $\det A_d(u) \neq 0$ 
for all $u \in \mathbb B(\tilde{u}, \delta)$.}

\end{remark}

\subsection{The continuity of the $\epsilon$-pseudospectrum set-valued map}

For a given parameter $u \in \mathbb{C}^p$ and a given $\epsilon \ge 0,$ the {\em $\epsilon$-pseudospectrum} of $P_{u}(\lambda)$ is defined to be
\begin{eqnarray*}
\Lambda_\epsilon(u) &:=& \left\{ \lambda \in \mathbb{C} \ | \ \det P_{v}(\lambda) = 0 \mbox{ for some } v \in \mathbb{B}(u, \epsilon) \right\}.
\end{eqnarray*}
By definition, $\Lambda_\epsilon(u)$ is the image of the semialgebraic set
\begin{eqnarray*}
\left\{ (\lambda, v) \in \mathbb{C} \times \mathbb{C}^p  \ | \ \det P_{v}(\lambda) = 0, \|v - u\| \le \epsilon \right\}
\end{eqnarray*}
via the projection map $\mathbb{C} \times \mathbb{C}^p \rightarrow \mathbb{C}, (\lambda, v) \mapsto \lambda,$ and so the $\epsilon$-pseudospectrum
$\Lambda_\epsilon(u)$ is a semialgebraic set because of the Tarski--Seidenberg theorem.

\begin{theorem}\label{thr:epsilon-pseudospectrum}
Let $\tilde{u} \in \mathbb{C}^p$ and $\tilde{\epsilon}, \delta > 0$ be given. Then there exist constants $c > 0$ and $\alpha > 0$ such that
\begin{eqnarray*}
\Lambda_{\epsilon}(u)  \subset \Lambda_{\epsilon'}(u') + c \big(\|u-u'\|^\alpha + |\epsilon - \epsilon'|^\alpha\big)\mathbb{B}
\end{eqnarray*}
for all $u, u' \in \mathbb B(\tilde{u},\delta)$ and all $\epsilon, \epsilon' \in [0, \tilde{\epsilon}].$
In particular, the $\epsilon$-pseudospectrum set-valued map
$$\Lambda_{\epsilon} \colon \mathbb{C}^p \rightrightarrows \mathbb{C}, \quad u \mapsto \Lambda_{\epsilon}(u),$$
is H\"older continuous of degree $\alpha$ at $\tilde{u}.$
\end{theorem}

\begin{proof}
By Theorem \ref{Theorem32}, there exist constants $c > 0$ and $\alpha  > 0$ such that
\begin{equation} \label{equ-pseudo}
 \sigma(v) \subset  \sigma({v'}) + c \|v-{v'}\|^\alpha \mathbb{B}
 \end{equation}
for all $  v, {v'} \in \mathbb B(\tilde{u}, \tilde{\epsilon} + \delta).$ Take any $u, u' \in \mathbb B(\tilde{u},\delta)$ and any $\epsilon, \epsilon' \in [0, \tilde{\epsilon}].$ Then the desired statement follows immediately from Claims~\ref{claim1} and \ref{claim2} below.
\end{proof}

\begin{claim} \label{claim1}
We have
$$\Lambda_{\epsilon}(u)  \subset \Lambda_{\epsilon'}(u) + c  |\epsilon-\epsilon'|^\alpha \mathbb{B}.$$
\end{claim}
\begin{proof}
If $\epsilon \leq \epsilon',$ then $\Lambda_\epsilon(u) \subset \Lambda_{\epsilon'}(u),$ which implies the claim. So we assume that $\epsilon >  \epsilon'.$

Take any $\lambda \in \Lambda_\epsilon({u}).$ By definition, there exists $v \in \mathbb B({u}, \epsilon)$ such that $\lambda \in \sigma(v).$  If $v\in \mathbb B(u, \epsilon')$, then $\lambda \in \Lambda_{\epsilon'}(u),$ and so there is nothing to prove. Therefore we assume that $v \not \in \mathbb B(u, \epsilon')$.  Let  $v'$ be the intersection point of the segment joining $u$ and $v$ with the sphere centered at $u$ of radius $\epsilon'.$ Then $v' \in \mathbb B({u}, \epsilon'),$ and hence $\sigma(v')\subset \Lambda_{\epsilon'}({u}).$  Moreover, it is clear that
\begin{eqnarray*}
\|v- v'\| & = & \|v - u\| - \|u - v'\| \ = \ \|v - u\| - \epsilon' \ \leq \ \epsilon - \epsilon'.
\end{eqnarray*}
On the other hand, we have $v, v'\in \mathbb B(\tilde{u}, \tilde{\epsilon} + \delta)$ because the following inequalities hold
\begin{eqnarray*}
\|v - \tilde{u}\| & \leq & \|v - u\| + \|u - \tilde{u}\| \ \leq \ \epsilon + \delta \ \le \ \tilde{\epsilon} + \delta, \\
\|v' - \tilde{u}\| & \leq & \|v' - u\| + \|u - \tilde{u}\| \ \leq \ \epsilon' + \delta \ \le \ \tilde{\epsilon} + \delta.
\end{eqnarray*}
Therefore, it follows from  \eqref{equ-pseudo} that
\begin{eqnarray*}
\mbox{dist}(\lambda, \Lambda_{\epsilon'}({u})) & \leq & \mbox{dist}(\lambda, \sigma(v')) \ \leq \ c \|v - v'\|^\alpha \ \leq \ c (\epsilon - \epsilon')^\alpha.
\end{eqnarray*}
Since $\lambda$ is arbitrary in $\Lambda_\epsilon({u})$, we get the required inclusion.
\end{proof}

\begin{claim} \label{claim2}
The following inclusion holds
$$\Lambda_{\epsilon'}(u)  \subset \Lambda_{\epsilon'}(u') + c  \|u-u'\|^\alpha \mathbb{B}.$$
\end{claim}

\begin{proof}
Let us take any  $\lambda \in \Lambda_{\epsilon'}(u)$. By definition, there exists $v \in \mathbb B(u, {\epsilon'})$ such that $\lambda \in \sigma(v)$.  If $v \in \mathbb B(u',{\epsilon'})$, then $\sigma(v)\subset \Lambda_{\epsilon'}(u'),$ which implies that  $\lambda \in \Lambda_{\epsilon'}(u'),$ and there is nothing to prove. Therefore we may assume that $v \not \in \mathbb B(u', {\epsilon'})$.   Let ${v'}$ be the intersection point of the segment joining $v$ and $u'$ with the sphere centered at $u'$ of radius ${\epsilon'}.$  Then $\sigma({v'})\subset \Lambda_{\epsilon'}(u')$, and we have
\begin{eqnarray*}
\|v-{v'}\|  & =  & \|v-u'\| -  \|u'-{v'}\| \\
 & \leq  &  (\|v-u\| + \|u-u'\|) - {\epsilon'} \\
&    \leq  & ({\epsilon'} +  \|u-u'\|) - {\epsilon'} \  =  \  \|u - u'\|.
\end{eqnarray*}
On the other hand, we have $v, {v'} \in \mathbb{B}(\tilde{u}, \tilde{\epsilon} + \delta)$ because it holds that
\begin{eqnarray*}
\|v - \tilde{u}\|  & \le  & \|v - u\| + \|u - \tilde{u}\| \ \le \ {\epsilon'} + \delta \ \le \ \tilde{\epsilon} + \delta, \\
\|{v'} - \tilde{u}\|  & \le  & \|{v'} - u'\| + \|u' - \tilde{u}\| \ \le \ {\epsilon'} + \delta \ \le \ \tilde{\epsilon} + \delta.
\end{eqnarray*}
Therefore, we deduce from \eqref{equ-pseudo} that
\begin{eqnarray*}
\mbox{dist}(\lambda, \Lambda_{\epsilon'}(u')) & \leq & \mbox{dist}(\lambda, \sigma({v'})) \ \leq \ c \|v - {v'}\|^\alpha \ \leq \ c \|u - u'\|^\alpha.
\end{eqnarray*}
Since $\lambda$ is arbitrary in $\Lambda_{\epsilon'}(u),$ the claim follows.
\end{proof}

If the matrix polynomial $P_{u}(\lambda)$ is real (i.e., all its coefficients are real) or Hermitian (i.e., all of its coefficients are Hermitian), then it is well-known that the spectrum of $P_{u}(\lambda)$ is symmetric with respect to the real axis. For the $\epsilon$-pseudospectra of matrix polynomials we have a similar result; see also \cite[Proposition~2.1]{Lancaster2005}.  The proof of this result is straightforward, hence we omit here.

\begin{proposition} \label{Proposition31}
Let ${u} \in \mathbb{C}^p$ and $\epsilon > 0.$ Assume that for all $v \in \mathbb{B}({u}, \epsilon),$ the matrices $A_k(v), ~ k = 0, \ldots, d, $ are real or Hermitian. Then the $\epsilon$-pseudospectrum $\Lambda_\epsilon({u})$ is symmetric with respect to the real axis.
\end{proposition}

 The next result is a structured analogue of the corresponding results in \cite[Theorem 2.3]{Lancaster2005} and \cite[Theorem~2]{Mosier1986}, obtained under parameter-dependent perturbations. The proofs follow the same arguments as in the classical case of standard pseudospectra; see, e.g. \cite[Theorem 2.3]{Lancaster2005}.
\begin{proposition}
Let ${u} \in \mathbb{C}^p$ and $\epsilon > 0.$ Then the $\epsilon$-pseudospectrum $\Lambda_\epsilon({u})$ has at most $dn $ connected components, and for any $v \in \mathbb{B}({u}, \epsilon),$ $P_{v}(\lambda)$ has an eigenvalue in each of these components. Furthermore, $P_{{u}}(\lambda)$ and $P_{v}(\lambda)$ have the same number of eigenvalues (counting multiplicities) in each connected component of $\Lambda_\epsilon({u}).$
\end{proposition}

The next statement shows that the boundary of the $\epsilon$-pseudospectrum is made up of algebraic curves. This is a comforting property in the sense that the number of difficult points, such as cusps or self-intersections, is limited. For related results, see also \cite{Boulton2008, Lancaster2005, Tisseur2001-1}.

\begin{proposition}
For each $\epsilon > 0,$ the boundary $\partial \Lambda_\epsilon({u})$ of the $\epsilon$-pseudospectrum $\Lambda_\epsilon({u})$ lies on an algebraic curve. In particular, $\partial \Lambda_\epsilon({u})$ has at most a finite number of singularities where the tangent fails to exist, and it intersects itself only at a finite number of points.
\end{proposition}

\begin{proof}
Indeed, the boundary of $\Lambda_\epsilon({u})$ is a semialgebraic set because of the Tarski--Seidenberg theorem.
Furthermore, it follows from Proposition~\ref{DimensionProposition} that
$$\dim \partial \Lambda_\epsilon({u}) < \dim \Lambda_\epsilon({u}) \le 2,$$
and so $\partial \Lambda_\epsilon({u})$ lies on an algebraic curve in $\mathbb{C} \equiv \mathbb{R}^2.$ By \cite[Theorem~1.1]{HP2017}, there is a finite partition of $\partial \Lambda_\epsilon({u})$ into points and analytic semialgebraic curves. The result follows.
\end{proof}

 \begin{remark} \label{rem:theorem33}\rm Theorem \ref{thr:epsilon-pseudospectrum} is still true if we replace the hypothesis on monicity of the matrix polynomial $P_u(\lambda)$ by the locally non-singularity of the leading coefficient matrix $A_d(u)$ as given in Remark \ref{rem:theorem32}.
\end{remark}

\subsection{The continuity of the numerical range set-valued map}

For any $u\in \mathbb C^p$, the {\em numerical range} of the monic matrix polynomial $P_{u}(\lambda)$ is defined as follows
\begin{eqnarray*}
W(u) := \{ \lambda \in \mathbb{C} \ | \ x^* P_{u}(\lambda) x = 0 \textrm{ for some } x \in \mathbb{C}^n, x^* x = 1\}.
\end{eqnarray*}
Observe that if $\lambda \in \sigma(u),$ then there exists a unit vector $x \in \mathbb{C}^n$ such that $x^* P_{u}(\lambda) x = 0.$ Hence $\sigma(u) \subset W(u).$

By definition, $W(u)$ is the image of the semialgebraic set
\begin{eqnarray*}
\{ (\lambda, x) \in \mathbb{C} \times \mathbb{C}^n \ | \ x^* P_{u}(\lambda) x = 0 \textrm{ and } x^* x = 1\}
\end{eqnarray*}
via the projection map $\mathbb{C} \times \mathbb{C}^n \rightarrow \mathbb{C}, (\lambda, x) \mapsto \lambda,$ and so it is a semialgebraic set because of the Tarski--Seidenberg theorem. Furthermore, we can see that the numerical range $W(u)$ is closed.
Since $P_u(\lambda)$ is a monic matrix polynomial, $W(u)$ is bounded and so it has at most $d$ connected components, where $d$ is the degree of $P_u(\lambda);$ see \cite[Theorems~2.2 and 2.3]{Li1994}.

\begin{lemma}\label{Lemma34}
Let $\tilde{u} \in \mathbb{C}^p.$ For   sufficiently small $\delta > 0$ there exists a compact $K \subset \mathbb{C}$ such that
$$W(u) \subset K \quad \textrm{ for all } \quad u \in \mathbb{B}(\tilde{u}, \delta).$$
\end{lemma}
\begin{proof}
The definition of the monic matrix polynomial $P_u(\lambda)$ implies that
\begin{eqnarray*}
x^* P_{u}(\lambda) x  &=& \left(x^* x\right) \lambda^d  + \left( x^* A_{d - 1}(u) x \right) \lambda^{d - 1}  + \cdots + \left( x^* A_{0}(u) x \right)
\end{eqnarray*}
is a monic polynomial in $\mathbb{C}[\lambda]$ whose coefficients depend continuously on $(x, u) \in \mathbb{C}^n \times \mathbb{C}^p.$ Consequently, the roots of  the equation $x^* P_{u}(\lambda) x = 0$ with $x^* x = 1$ are continuous functions in $(x, u).$ Then the desired result follows  from the compactness of the set $\{x \in \mathbb{C}^n \ | \ x^* x = 1\} \times \mathbb{B}(\tilde{u}, \delta).$
\end{proof}

\begin{theorem}\label{Theorem33}
Let $\tilde{u} \in \mathbb{C}^p.$ For sufficiently small $\delta > 0$, there exist constants $c > 0$ and $\alpha > 0$ such that
\begin{eqnarray*}
W(u) &\subset& W(u') + c \|u - u'\|^\alpha \mathbb{B} \quad \textrm{ for all } \quad u, u' \in \mathbb{B}(\tilde{u}, \delta).
\end{eqnarray*}
In particular, the numerical  set-valued map
$$W \colon \mathbb{C}^p \rightrightarrows \mathbb{C}, \quad u \mapsto W(u),$$
is H\"older continuous of degree $\alpha$ at $\tilde{u}.$
\end{theorem}

\begin{proof}
Fix $\delta > 0.$ Then for any $u\in \mathbb{B}(\tilde{u}, \delta)$ and any $\lambda \in W(u)$, by definition, there exists a unit vector $x\in \mathbb C^n$ such that $x^*P_u(\lambda)x =0$. For each  $u'\in \mathbb{B}(\tilde{u}, \delta)$, the function
$$ \mathbb C \rightarrow \mathbb C, \quad  z\mapsto x^*P_{u'}(z)x,$$
is a monic polynomial function of degree $d $ in the variable $z$ whose  coefficients are  $x^*A_k(u')x, ~k=0,\ldots, d,$ which are continuous functions in $u'.$ We can write
$$ x^* P_{u'}(z)x = \prod_{k = 1}^d (z - \lambda_k(u')),$$
where $\lambda_k(u')\in \mathbb C$. Note also that
$$ \lambda_k(u') \in W(u') \quad \textrm{ for all } \quad k = 1,\ldots,d.$$
Therefore
\begin{eqnarray*}
\mathrm{dist}(\lambda, W(u')) & \leq & \min_{k = 1,\ldots,d} |\lambda - \lambda_k(u')|\\
& \leq & \left(\prod_{k = 1}^d |\lambda - \lambda_k(u')| \right)^{\frac{1}{d}} \ = \ \big |x^*P_{u'}(\lambda) x \big |^{\frac{1}{d}} \\
& = & \big |x^*P_{u}(\lambda) x-x^*P_{u'}(\lambda) x \big |^{\frac{1}{d}} \ = \ \big |x^*(P_{u}(\lambda)-P_{u'}(\lambda))x \big |^{\frac{1}{d}}\\
& \leq & \bigg(\|P_{u}(\lambda)-P_{u'}(\lambda)\| \cdot \|x\|^2\bigg)^{\frac{1}{d}} \ =  \ \big \|P_{u}(\lambda)-P_{u'}(\lambda) \big \|^{\frac{1}{d}}.
\end{eqnarray*}
Then the proof follows from Lemmas \ref{Lemma32} and \ref{Lemma34}.
\end{proof}

\begin{corollary}\label{coro31}
Let $\tilde{u} \in \mathbb{C}^p.$ For  sufficiently small $\delta > 0$ there exist constants $c > 0$ and $\alpha > 0$ such that
\begin{eqnarray*}
\partial W(u) &\subset& \partial W(u') + c \|u - u'\|^\alpha \mathbb{B} \quad \textrm{ for all } \quad u, u' \in \mathbb{B}(\tilde{u}, \delta).
\end{eqnarray*}
In particular, the set-valued map
$$\partial  W \colon \mathbb{C}^p \rightrightarrows \mathbb{C}, \quad u \mapsto \partial W(u),$$
is H\"older continuous of degree $\alpha$ at $\tilde{u}.$
\end{corollary}

\begin{proof}
This is an immediate consequence of Theorem~\ref{Theorem33}.
\end{proof}

\begin{remark}\rm (i) 
If the matrix polynomial $P_{{u}}(z)$ is real or Hermitian, then its numerical range $W(u)$ is symmetric with respect to the real axis.
The proof is similar to that of Proposition~\ref{Proposition31}, and so we omit the details.

(ii) For each $u \in \mathbb{C}^p,$ the boundary of the numerical range, $\partial W(u),$ is a semialgebraic set because of Tarski--Seidenberg's theorem.
Furthermore, it follows from Proposition~\ref{DimensionProposition} that
$$\dim \partial W(u) < \dim W(u) \le 2,$$
and so $\partial W(u)$ lies on an algebraic curve. By \cite[Theorem~1.1]{HP2017}, there is a finite partition of $\partial W(u)$ into points and analytic semialgebraic curves. Consequently, $\partial W(u)$ has at most a finite number of singularities where the tangent fails to exist, and it intersects itself only at a finite number of points. Furthermore, thanks to \cite[Theorem~1.4]{HP2017}, there exists an integer $N$ such that the number of nondifferentiable points of $\partial W(u)$ is bounded above by $N$ for all $u.$ See also \cite{Jonckheere1998, Maroulas1997} for some algebraic and topological properties of the nondifferentiable points of the boundary of the numerical range of matrix polynomials.

 (iii)  Theorem \ref{Theorem33} is still true if we consider arbitrary matrix polynomial $P_{u}(\lambda) = A_d(u) \lambda^d  + A_{d - 1}(u) \lambda^{d - 1}  + \cdots + A_{0}(u)$ satisfying the  following condition:
$$
\inf_{\|x\|=1} |x^* A_d(u) x| \geq \gamma > 0 \quad \text{for } u \in B(\tilde{u}, \delta).
$$
A sufficient condition should be added is to assume $A_d(u)$ Hermitian positive definite (uniformly in the neighborhood).
\end{remark}

\subsection{The continuity of the joint numerical range set-valued map}

Recall that for any $u\in \mathbb C^p$, the {\em joint numerical range} of  the $(d + 1)$-tuple of the matrices $A_0(u), \ldots, A_{d}(u)$ is defined as
\begin{eqnarray*}
JW(u) := \{(x^* A_0(u) x, \ldots, x^*A_{d}(u)x) \ | \   x \in \mathbb{C}^n, x^*x=1\}.
\end{eqnarray*}
The joint numerical range, being a continuous image of the unit sphere, is compact and connected but not necessarily convex; see \cite{Binding1991}. By the Tarski--Seidenberg theorem, $JW(u)$ is a semialgebraic set, in particular, it has a finite number of connected components. Furthermore, observe from definition that
\begin{eqnarray*}
W(u) &=& \{ \lambda \in \mathbb{C} \ | \ a_d \lambda^d + a_{d - 1} \lambda^{d - 1} + \cdots + a_0 = 0,  \quad (a_{0}, \ldots, a_d) \in JW(u) \}.
\end{eqnarray*}

\begin{theorem}\label{Theorem34}
Let
$P_{u}(\lambda) := A_d(u) \lambda^d  + A_{d - 1}(u) \lambda^{d - 1}  + \cdots + A_{0}(u)$
be a matrix polynomial.  Let $\tilde{u} \in \mathbb{C}^p$. Then, for  sufficiently small $\delta > 0$ there exist constants $c > 0$ and $\alpha > 0$ such that
\begin{eqnarray*}
JW(u) &\subset& JW(u') + c \|u - u'\|^\alpha \mathbb{B} \quad \textrm{ for all } \quad u, u' \in \mathbb{B}(\tilde{u}, \delta).
\end{eqnarray*}
In particular, the joint numerical range set-valued  map
$$JW \colon \mathbb{C}^p \rightrightarrows  \mathbb{C}, \quad u \mapsto JW(u),$$
is H\"older continuous of degree $\alpha$ at $\tilde{u}.$
\end{theorem}

\begin{proof}
Fix $\delta > 0.$ By Lemma~\ref{Lemma31}, there exist constants $c > 0$ and $\alpha > 0$ such that
for all $u, u'\in \mathbb{B}(\tilde{u}, \delta)$ it holds that
$$\max_{k = 0, \ldots, d} \|A_k(u) - A_k(u')\| \le c\|u - u'\|^\alpha.$$

Take any $u, u'\in \mathbb{B}(\tilde{u}, \delta)$ and any $y \in JW(u).$ By definition, there exists a vector $x\in \mathbb C^n$ with $x^*x = 1$ such that
$$y = (x^* A_0(u) x, \ldots, x^*A_{d}(u)x).$$
Let
$$y' := (x^* A_0(u') x, \ldots, x^*A_{d}(u')x) \in JW(u').$$
We have
\begin{align*}
\mathrm{dist}(y, JW(u')) & \leq   \|y-y'\|\\
& \leq \max_{k = 0, \ldots, d} \big|x^*\big(A_k(u) - A_k(u')\big)x\big|\\
& \leq \max_{k = 0, \ldots, d} \|A_k(u) - A_k(u')\| \|x\|^2 \ \le \  c\|u - u'\|^\alpha.
\end{align*}
This gives the desired statement.
\end{proof}

\subsection{The stability of the spectral radius}

The {\em spectral radius}   of the matrix polynomial $P_u(\lambda)$ is defined by
$$\mathfrak{r}(u) := \max \{|\lambda|  \ | \  \lambda \in \sigma (u)\}.$$
The stability of the spectral radius is given as follows.

\begin{theorem}  Let
$P_{u}(\lambda) := A_d(u) \lambda^d  + A_{d - 1}(u) \lambda^{d - 1}  + \cdots + A_{0}(u)$
be a matrix polynomial. 
Let $\tilde{u} \in \mathbb{C}^p.$ For  sufficiently small $\delta > 0$ there exist constants $c > 0$ and $\alpha > 0$ such that
\begin{eqnarray*}
|\mathfrak{r}(u) - \mathfrak{r}(u')| & \le & c \|u - u'\|^\alpha \quad \textrm{ for all } \quad u, u' \in \mathbb{B}(\tilde{u}, \delta).
\end{eqnarray*}
\end{theorem}

\begin{proof}
By assumption, $\det P_u(\lambda) \in \mathbb{C}[\lambda]$ is a monic polynomial whose coefficients depend continuously on the parameter $u \in \mathbb{C}^p.$
Hence, all the roots of the equation
$$\det P_u(\lambda) = 0$$
depend continuously on the parameter $u.$ It follows that the function $u \mapsto \mathfrak{r}(u)$ is continuous. On the other hand, by the Tarski--Seidenberg theorem, we can see that the function $$\mathbb{C}^p \to \mathbb{R}, \quad u \mapsto \mathfrak{r}(u),$$
is semialgebraic. Then the desired conclusion follows from Proposition~\ref{LojasiewiczInequality}.
\end{proof}

\subsection{The genericity of set-valued maps}

The next result shows that, in general, the set-valued maps $\sigma, \Lambda_{\epsilon}$, $W,$ $\partial W$ and  $JW$ are H\"older continuous with exponents explicitly determined.

\begin{theorem}\label{Theorem41}
There exists an open and dense semialgebraic set $U \subset \mathbb{C}^p$ satisfying the following conditions:

\begin{enumerate}
\item[(i)]  The restriction of the set-valued maps $\sigma, \Lambda_{\epsilon}$, $W$ and $\partial W$ on $U$ are H\"older continuous of degree $\frac{1}{nd}, \frac{1}{nd}$, $\frac{1}{d}$ and $\frac{1}{d},$ respectively.

\item[(ii)] The restriction of the set-valued map $JW$ on $U$ is Lipschitz continuous.

\item[(iii)] The restriction of the spectral radius $\mathfrak{r}$ on $U$ is analytic.
\end{enumerate}

\end{theorem}

\begin{proof}
By ~\cite[Theorem~1.7]{HP2017}, there exists an open and dense semialgebraic set $U \subset \mathbb{C}^p$ such that the restriction
of the spectral radius $\mathfrak{r}$ and of the matrices $A_0(u), \ldots, A_{d}(u)$ on $U$ are analytic.
In particular, the item~(iii) follows.

Take any $\tilde{u} \in U.$ There exists $\delta > 0$ such that $\mathbb{B}(\tilde{u}, \delta) \subset U.$ Then we can find a constant $c > 0$ such that
\begin{eqnarray*}
\max_{k = 0, \ldots, d} \|A_k(u) - A_k(u')\| & \le & c \|u - u'\| \quad \textrm{ for all } \quad u, u' \in \mathbb{B}(\tilde{u}, \delta).
\end{eqnarray*}
In other words, we can let $\alpha = 1$ in Lemmas~\ref{Lemma31}~and~\ref{Lemma32}.
Finally, the items (i) and (ii) follow immediately from the proofs of Theorems~\ref{Theorem32}, \ref{Theorem33}, \ref{Theorem34} and
Corollary~\ref{coro31}.
\end{proof}

\section{Stability of Jordan pairs of monic matrix polynomials} \label{jordanpair}

Throughout this section we study the stability of the  Jordan pairs of the  monic matrix polynomial
$$P_{u}(\lambda) := I \lambda^d  + A_{d - 1}(u) \lambda^{d - 1}  + \cdots + A_{0}(u),$$
where   $A_{k}(u), k = 0, \ldots, d - 1,$ are complex $n \times n$ matrices whose entries are continuous semialgebraic functions in $u \in \mathbb{C}^p.$

\subsection{Jordan pairs of monic matrix polynomials}
Firstly we recall  the definition of the Jordan pair of the monic matrix polynomial
$$P(z) := I  z^d  + A_{d - 1}  z^{d - 1}  + \cdots + A_{0},$$
 where $A_k \in \mathbb C^{n \times n}$ for all  $k = 0, \ldots, d - 1$.  This notation can be seen in the book \cite{Gohberg1982}.

Let $\lambda_0$ be an eigenvalue of $P(z)$.

\begin{definition}\rm
An $n$-dimensional vector polynomial $\varphi(\lambda)$ such that $\varphi(\lambda_0)\not = 0$ and $P(\lambda_0)$ $\varphi(\lambda_0)$ $  =0$ is called a \emph{root polynomial } of $P(z)$ corresponding to $\lambda_0$. The order of $\lambda_0$ as a zero of
$P(\lambda_0)\varphi(\lambda_0)$ is called the \emph{order} of the root polynomial $\varphi(\lambda)$.

Let $\varphi_1 (\lambda)$ be a root polynomial of $P(z)$ corresponding to $\lambda_0$ with the largest order $k_1$.  Then we may write
$$ \varphi_1(\lambda)=\displaystyle\sum_{j=0}^{k_1-1} (\lambda - \lambda_0)^j \varphi_{1j},$$
and $\varphi_{10}, \ldots, \varphi_{1 k_1-1}$ form a \emph{Jordan chain} of $P(z)$ corresponding to $\lambda_0$, i.e. for each $i = 0, \ldots, k_1-1$, we have
$$ \sum_{j=0}^i \dfrac{1}{j!}P^{(j)}(\lambda_0)\varphi_{1 i-j}=0,$$
where $P^{(j)}(\lambda_0)$ denotes the $j$th derivative of the matrix polynomial $P(z)$ with respect to . $z$ at the point $\lambda_0$.

Let  $\varphi_2(\lambda)=\displaystyle\sum_{j=0}^{k_2-1} (\lambda - \lambda_0)^j \varphi_{2j}$ be a  root polynomial of $P(z)$   with the largest order among all the root polynomials whose eigenvector is not a scalar multiple of $\varphi_{10}$ $(k_2 \leq k_1)$.

Assume that $\varphi_1, \ldots, \varphi_{s-1}$ have already chosen, and we can  write
$$ \varphi_i(\lambda)= \displaystyle\sum_{j=0}^{k_i-1} (\lambda - \lambda_0)^j \varphi_{ij}, ~ i=0, \ldots, s-1,$$
let us define
$$ \varphi_s(\lambda)= \displaystyle\sum_{j=0}^{k_s-1} (\lambda - \lambda_0)^j \varphi_{sj} $$
a root polynomial of $P(z)$ where $k_s$ is the largest order among all the root polynomials whose eigenvector is not in the span of the eigenvectors $\varphi_{10}, \ldots,\varphi_{s-1 0}$. Continue this process until the set $\mbox{Ker} P(\lambda_0)$ of eigenvectors corresponding to $\lambda_0$ is exhausted. Therefore we have already constructed $r$ root polynomials
$$ \varphi_i(\lambda)= \displaystyle\sum_{j=0}^{k_i-1} (\lambda - \lambda_0)^j \varphi_{ij}, ~ i=0, \ldots, r,$$
where $r := \dim \mbox{Ker} P(\lambda_0)$. The Jordan chains
$$\varphi_{10}, \ldots, \varphi_{1 k_1-1}, \quad \varphi_{20}, \ldots, \varphi_{2 k_2-1}, \quad  \ldots, \quad \varphi_{r0}, \ldots, \varphi_{r k_r-1}  $$
is said to form  a \emph{canonical set of Jordan chains} of $P(z)$ corresponding to the eigenvalue $\lambda_0$.
\end{definition}

Now let $\lambda_1, \ldots, \lambda_r$  be the distinct eigenvalues of $P(z)$, and let $\alpha_i$ be the multiplicity of $\lambda_i,$ $i=1, \ldots, r$. It is clear that $\displaystyle\sum_{i=1}^r \alpha_i=mn$.
\begin{definition} \rm For each $\lambda_i$, let us choose a canonical set of Jordan chains of $P(z)$ corresponding to $\lambda_i$:

$$\varphi_{j0}^{(i)}, \ldots, \varphi_{j k_j^{(i)}-1}^{(i)}, \quad j=0, \ldots,s_i. $$
It should be noted that $\displaystyle\sum_{j=1}^{s_i} k_j^{(i)} = \alpha_i$.

Denote
$$X_i:=\big[\varphi_{10}^{(i)}  \cdots \varphi_{1 k_1^{(i)}-1}^{(i)}, \varphi_{20}^{(i)}  \cdots \varphi_{2 k_2^{(i)}-1}^{(i)}, \cdots, \varphi_{s_i 0}^{(i)}  \cdots \varphi_{s_i k_{s_i}^{(i)}-1}^{(i)}\big] $$
which is a matrix of size $n\times \left(\sum_{j=1}^{s_i} k_j^{(i)}\right) = n\times \alpha_i$.

For each $j=1, \ldots, s_i$, denote by $J_{ij}$ the Jordan block of size $k_j^{(i)} \times k_j^{(i)}$ corresponding to the eigenvalue $\lambda_i.$ We denote by $J_i$ the block-diagonal
$$ J_i=\left[\begin{array}{cccc}
J_{i1} & & & 0\\
       &J_{i2}& & \\
       & & \ddots & \\
   0   & & & J_{is_i}
\end{array}\right]. $$
Then the pair $(X_i,J_i)$ constructed as above is called a {\em Jordan pair of $P(z)$ corresponding to $\lambda_i$}.

Denote
$$X:=\big[X_1, \ldots, X_r \big], \quad J:=\big[J_1,\ldots, J_r\big]. $$
Then the pair $(X,J)$ is called a {\em Jordan pair} of the monic matrix polynomial $P(z)$.
\end{definition}

\subsection{Stability of Jordan pairs}

Now we return to consider the monic matrix polynomial
 $$P_{u}(\lambda) := I \lambda^d  + A_{d - 1}(u) \lambda^{d - 1}  + \cdots + A_{0}(u), ~ u \in \mathbb C^p.$$
For each $u \in \mathbb{C}^p$, let  $(X_u, J_u)$ be a Jordan pair of  $P_u(z)$, where
$$ X_u=\big[X_u^{(1)}, \ldots, X_u^{(r)}\big], \quad J_u := \mathrm{diag}\big[J_u^{(1)}, \ldots, J_u^{(r)})\big],$$
and each $(X_u^{(k)}, J_u^{(k)})$, $k = 1, \ldots, r,$ is a Jordan pair of $P_u(z)$ corresponding to the eigenvalue $\lambda_i(u)$, where $\lambda_1(u), \ldots, \lambda_r(u)$ are distinct eigenvalues of $P_u(z).$

The following theorem is inspired by the work of  I.~Gohberg, P.~Lancaster and L.~Rodman \cite[Proposition~5.10 and Theorem~5.11]{Gohberg1982}.

\begin{theorem}\label{Theorem42}
There is a finite partition of $\mathbb{C}^p$ into analytic semialgebraic submanifolds $M_{1}, \ldots, M_N$ such that for each $t = 1, \ldots, N,$ the following statements hold:
\begin{enumerate}
\item[(i)] For every $u \in M_t$ the number $r$ of Jordan blocks in $J_u$ and their sizes $q_1, \ldots, q_r$ are independent of $u.$

\item[(ii)] The eigenvalues $\lambda_k(u), k = 1, \ldots, r,$ are analytic functions of $u$ on $M_{t}.$

\item[(iii)] The blocks $X_u^{(k)}, k = 1, \ldots, r,$ of $X_u$ are analytic functions of $u$ on $M_{t}.$
\end{enumerate}
\end{theorem}

\begin{proof}
Passing to the linearization (with companion matrix) we can suppose $P_u(\lambda)$ is linear, i.e., $P_u(\lambda) = I \lambda - C(u),$ where
\begin{eqnarray*}
C(u) &:=&
\begin{pmatrix}
0 & I & 0 & \cdots & & 0 \\
0 & 0 & I & \cdots & & 0 \\
\vdots & \vdots & \vdots & \vdots & & \vdots \\
0 & 0 & 0 & \cdots & & I \\
- A_0(u) & - A_1(u) &  - A_2(u) & \cdots &  & -A_{d - 1}(u)
\end{pmatrix}
\end{eqnarray*}
is the companion matrix of $P_u(\lambda).$ For given $k = 1, \ldots, dn,$ denote
\begin{eqnarray*}
Z_k &:=& \{u \in \mathbb{C}^p \ | \ \det \big (I \lambda - C(u) \big)  = 0 \textrm{ has exactly $k$ different roots} \}.
\end{eqnarray*}
Then the sets $Z_k$ are semialgebraic (by the Tarski--Seidenberg theorem) and $\mathbb{C}^p = \displaystyle\bigcup_{k = 1}^{dn} Z_k.$

Fix $k$ and assume that $Z_k \ne \emptyset.$ From the construction of $Z_k$ it follows that the number of different
eigenvalues of $I \lambda - C(u)$ is constant $k$ for $u \in Z_k.$ So let $\lambda_1(u), \ldots, \lambda_k(u), u \in Z_k,$ be all the different eigenvalues of $I \lambda - C(u).$
For all $i = 1, \ldots, k,$ and all $j, l  = 1, 2, \ldots, dn,$ the sets
\begin{eqnarray*}
\{u \in Z_k & | & \mathrm{rank} \big (I \lambda_i(u) - C(u) \big )^j = l \}
\end{eqnarray*}
are semialgebraic  by the Tarski--Seidenberg theorem again, since rank constraints can be expressed by polynomial equations and inequalities involving matrix minors, and it is clear that $Z_k$ is the disjoint union of these sets. Furthermore, thanks to~\cite[Theorem~1.7]{HP2017}, we may decompose $Z_k$ onto a finite disjoint union of analytic semialgebraic submanifolds of $\mathbb{C}^p$ such that the restriction of the eigenvalue functions
$\lambda_{k}(u)$ and of the blocks $X_u^{(k)}$ for $k = 1, \ldots, r,$ on each submanifold are analytic.

In summary, we can write $\mathbb{C}^p = \displaystyle\bigcup_{t = 1}^N M_t,$ where $M_t$ are nonempty analytic semialgebraic submanifolds satisfying the following conditions:
\begin{enumerate}
\item[(a)] $M_t \cap M_s = \emptyset$ for $t \ne s;$
\item[(b)] the number of different eigenvalues of $I \lambda - C(u)$ is constant for $u \in M_t;$
\item[(c)] for every eigenvalue $\lambda_0(u)$ (which is analytic on $M_t$) of $I \lambda - C(u),$
\begin{eqnarray*}
r({u; j}) &:=& \mathrm{rank} \big (I \lambda_0(u) - C(u) \big )^j
\end{eqnarray*}
is constant for $u \in M_t$ and $j = 1, 2, \ldots.$ The sizes of the Jordan blocks in the Jordan normal form of $C(u)$ corresponding to $\lambda_0(u)$ are completely determined by the numbers
$r(u; j);$ namely,
\begin{eqnarray} \label{PT5}
n - r({u; j}) &=& \gamma_1(u) + \cdots + \gamma_j(u) \quad \textrm{ for } \quad j = 1, 2, \ldots, dn,
\end{eqnarray}
where $\gamma_j(u)$ is the number of the Jordan blocks corresponding to $\lambda_0(u)$ whose size is not less than $j$ (so $\gamma_1(u)$ is just the number of the Jordan blocks). Equality~\eqref{PT5} is easily observed for the Jordan normal form of $C(u);$ then it clearly holds for $C(u)$ itself. Since $r(u; j)$ are constant for $u \in M_t,$
\eqref{PT5} ensures that the Jordan structure of $C(u)$ corresponding to $\lambda_0(u)$ is also constant for $u \in M_t.$
\end{enumerate}
The proof is complete.
\end{proof}

\begin{remark} \label{rem:theorem42}\rm Theorem \ref{Theorem42} is still true if we replace the hypothesis on monicity of the matrix polynomial $P_u(\lambda)$ by the locally non-singularity of the leading coefficient matrix $A_d(u)$ as given in Remark \ref{rem:theorem32}.
\end{remark}


\bibliographystyle{apalike}
\bibliography{MPB}
\end{document}